\documentclass[10pt,reqno]{amsart}

\usepackage{amssymb,amsfonts,amsmath,epsfig,amsthm}

 \allowdisplaybreaks
\numberwithin{equation}{section}

\font\script=rsfs10 at 11pt

\def\H{{\mbox{\script H}\,\,}}
\def\L{{\mbox{\script L}\,\,}}
\def\N{\mathbb N}
\def\R{\mathbb R}
\def\S{\mathbb S}
\def\C{\mathcal C}
\def\BigC{{\mbox{\script C}\,}}

\def\eps{\varepsilon}
\def\res{\mathop{\hbox{\vrule height 7pt width .5pt depth 0pt \vrule height .5pt width 6pt depth 0pt}}\nolimits}
\def\step#1#2{\par\noindent{\underline{\it Step~#1.}}\emph{ #2}\\}
\def\spt{\operatorname{spt}}
\def\Chi#1{\hbox{\Large$\chi$}_{#1}}
\def\freccia#1{\xrightarrow[\ #1]{}}
\def\proofof#1{\begin{proof}[Proof of~#1]}

\newtheorem{theorem}{Theorem}[section]
\newtheorem{lemma}[theorem]{Lemma}
\newtheorem{definition}[theorem]{Definition}

\newcounter{mt}
\def\maintheorem#1#2{\par \medskip \noindent {\bf Theorem~\mref{#1}}.~{\it #2}\par}
\def\mref#1{\Alph{#1}}
\def\maintheoremdeclaration#1{\stepcounter{mt}\newcounter{#1}\setcounter{#1}{\arabic{mt}}}

\maintheoremdeclaration{Main}

\begin{document}

\title[Continuity of the total cost with relativistic cost functions]{On the continuity of the total cost in the mass transport problem with relativistic cost functions}
\author{Jean Louet}
\address{CEREMADE, Universit\'e Paris-Dauphine \\ Place du Mar\'echal de Lattre de Tassigny, 75~775 Paris cedex 16, France}  
\address{INRIA, MOKAPLAN, 2 rue Simone Iff \\ 75 012 Paris, France}
\email{louet@ceremade.dauphine.fr}
\author{Aldo Pratelli}
\author{Florian Zeisler}
\address{Department Mathematik, Friederich-Alexander Universit\"at Erlangen-N\"urnberg \\
Cauerstrasse,11, 91058 Erlangen (Germany)}
\email{pratelli@math.fau.de \\ zeisler@math.fau.de}

\begin{abstract}
In this paper we consider the mass transport problem in the case of a relativistic cost; we can establish the continuity of the total cost, together with a general estimate about the directions in which the mass can actually move, under mild assumptions. These results generalize those of the recent paper~\cite{bpp}, also positively answering some of the open questions there.
\end{abstract}

\maketitle

\section{Introduction}

In this paper we concentrate on a particular question in the mass transport problem. The general mass transport problem in $\R^d$, which is now widely known (for a wide source on that, refer on the book~\cite{vil}), consists in considering two probability measures $\mu\neq \nu$ in $\R^d$, and trying to minimize the cost of the transport plans between $\mu$ and $\nu$. More precisely, the set $\Pi(\mu,\nu)$ of the transport plans is given by all the measures $\gamma$ on $\R^d\times \R^d$ whose marginals on the two copies of $\R^d$ are $\mu$ and $\nu$ respectively. And the cost of a transport plan $\gamma$ is given by
\[
\iint_{\R^d\times\R^d} c(x,y)\, d\gamma(x,y)\,,
\]
where $c:\R^d\times\R^d\to [0,+\infty]$ is some given l.s.c. function, called \emph{cost function}. It is immediate to show that the set of transport plans is never empty, and in particular there exist always minimizers of the cost, which are called \emph{optimal transport plans}.\par

A quite interesting example of a cost function is the so-called \emph{relativistic heat cost}, first introduced by Brenier in~\cite{Br}, and which is defined as $c(x,y)=h(y-x)$, where
\begin{equation}\label{brencost}
h(z)=\left\{\begin{array}{ll}
1-\sqrt{1-|z|^2} &|z|\leq 1\,,\\
+\infty&|z|>1\,.
\end{array}\right.
\end{equation}
Notice that the function $h$ is strictly convex in a strictly convex subset of $\R^d$ (the closed unit ball), and $+\infty$ outside. The study of the transport with relativistic heat cost was again studied in~\cite{McCaPu}, and generalized to the case of \emph{relativistic cost functions} in~\cite{bp}. The relativistic cost functions are defined again as $c(x,y)=h(y-x)$, but this time $h$ is a generic function which is strictly convex and bounded in a strictly convex and bounded subset $\BigC$ of $\R^d$, and $+\infty$ outside, see Definition~\ref{defrelco}. Actually, when speaking about the transport problem with a relativistic cost, there is an additional parameter $t>0$, corresponding to the time; more precisely, the relativistic cost functions are the functions $c_t$, for every $t>0$, which are defined as
\[
c_t(x,y) = h\bigg(\frac{y-x}t\bigg)\,.
\]
Notice that, if $t$ is very small, then $c_t(x,y)=+\infty$ for all the points $x,\,y$ which are not extremely close to each other, so it is easy to guess that all the tranport plans have infinite cost: this corresponds to the fact that, in a relativistic context, particles cannot move faster than a given maximal velocity, hence in a very short time it is simply impossible to transport the density $\mu$ onto $\nu$; instead, for $t$ bigger and bigger, not only it is possible to transform $\mu$ onto $\nu$, but it becomes also cheaper and cheaper. In particular, the following results were proved in~\cite{bp}; here, and through the whole paper, by $\C_t:\Pi(\mu,\nu)\to [0,\infty]$ we denote the \emph{cost relative to the time $t$}, and by $\C:(0,+\infty)\to [0,+\infty]$ the \emph{minimal cost} corresponding to the time $t$, which are given by
\begin{align*}
\C_t(\gamma) = \iint_{\R^d\times\R^d} c_t(x,y)\, d\gamma(x,y)\,, &&
\C(t) = \min \Big\{ \C_t(\gamma):\, \gamma\in\Pi(\mu,\nu)\Big\}\,.
\end{align*}
\begin{theorem}\label{first}
Let $\mu\neq \nu$ be two probability measures with compact support. Then there exists a time $T>0$, called \emph{critical time}, such that $\C(t)=+\infty$ for every $0<t<T$, while $\C(t)$ is bounded for every $t\geq T$; and moreover, the function $\C$ is non-increasing and right-continuous on the interval $[T,+\infty)$. Finally, there exists a unique optimal transport plan $\gamma_t$ for every $t>T$.
\end{theorem}

The study of the transport problem in a relativistic context has been then continued in the very recent paper~\cite{bpp}. There, the authors have introduced a subclass of the relativistic cost functions, namely, the \emph{highly relativistic cost functions}, see Definition~\ref{defhighrelco}. Basically, a relativistic cost is called ``highly relativistic'' if the slope of $h$ explodes on the boundary $\partial\BigC$ of the convex set where $h<\infty$. The reason to introduce this subclass is simple: observe that the original relativistic heat cost defined in~(\ref{brencost}) is highly relativistic; one can also notice that some of the nice properties which hold in this model case actually depend on the infinite slope of $h$ at the boundary. In fact, in the paper~\cite{bpp} the following results were proved.
\begin{theorem}\label{Thbpp}
Let $\mu\neq\nu$ be two probability measures with compact support in $\R^d$, and assume that $c_t$ is a highly relativistic cost function, and that $\mu\ll \L$, being $\L$ the Lebesgue measure on $\R^d$. Then,
\begin{itemize}
\item[(i)] If $\mu\in L^\infty$, then the function $t\mapsto\C(t)$ is continuous on $[T,+\infty)$.
\item[(ii)] For every supercritical time $t>T$, the optimal plan $\gamma_t$ satisfies
\[
\gamma_t\bigg(\bigg\{(x,y)\in\R^d\times\R^d:\, \frac{y-x}t \in \partial \BigC\bigg\}\bigg)=0\,.
\]
\end{itemize}
\end{theorem}
Let us now briefly discuss the above result, and then its assumptions. The first claim simply states the continuity of the total cost; the second one is basically saying that the optimal transport does not move the points ``with maximal distance'', that is, the vectors lying in the boundary of $\BigC$ are almost never used: notice that the cost, on the boundary of $\BigC$, is still bounded, so it is not obvious that it should not be convenient to use also such vectors
.\par

Let us now pass to discuss the assumptions of the two claims in Theorem~\ref{Thbpp}: while the fact that $\mu$ is bounded for the first claim is only a technical assumption, just helping to simplify the notations, the ``serious'' assumptions are that $\mu$ must be absolutely continuous with respect to the Lebesgue measure, and that the cost is not just relativistic, but also highly relativistic. The counterexamples in~\cite{bpp} ensure that both the claims are false without the absolute continuity, as well as the second claim without the assumption that $c_t$ is highly relativistic. Instead, the question whether the highly relativistic assumption is essential also for~(i) was set as an open question at the end of that paper.\par

The aim of this paper is to generalize Theorem~\ref{Thbpp} as much as possible. In particular, we will show that for the relativistic cost functions (not necessarily highly relativistic, then) the first claims is still valid (even removing the boundedness assumption for $\mu$), as well as a generalisation of the second one, which basically says that an optimal transport plan does not use those vectors in the boundary $\BigC$ where $h$ has infinite slope (which form the whole $\partial\BigC$ if the cost is highly relativistic). More precisely, our main result reads as follows.
\maintheorem{Main}{Let $\mu\neq\nu$ be two probability measures with compact support in $\R^d$, and assume that $c_t$ is a relativistic cost function, and that $\mu\ll \L$. Then,
\begin{itemize}
\item[(i)] The function $t\mapsto\C(t)$ is continuous on $[T,+\infty)$.
\item[(ii)] For every supercritical time $t>T$, the optimal plan $\gamma_t$ satisfies $\gamma_t \big(\big\{(x,y):\, \frac{y-x}t \in\Theta\big\}\big)=0$, where
\[
\Theta := \Big\{v \in \partial\BigC:\, D_{-v}h(v)=-\infty \Big\}\,,
\]
and $D_{-v} h(v) \in [-\infty,+\infty)$ is the slope of $h$ at the point $v$ in the direction $-v$.
\end{itemize}}

Our constructions are reminiscent of those made in~\cite{bpp}, but there are some fundamental differences and new ideas, which are necessary in order to deal with the quite weaker assumption of Theorem~\mref{Main} with respect to Theorem~\ref{Thbpp}, in particular with the fact of considering relativistic, but not highly relativistic cost functions.

The plan of the paper is simple: in Section~\ref{secprel} we list all the relevant definitions, notations, and the known technical facts, among which the so-called Chain Lemma (Lemma~\ref{chle}). Then, in the Sections~\ref{seccont} and~\ref{seclocat} we prove the two parts of our Theorem~\mref{Main}.

\section{Notation and preliminary results\label{secprel}}

In this section we collect some standard notation about the mass transport problem, and we give the relevant definitions about relativistic and highly relativistic cost functions.\par

Let $X$ and $Y$ be two Polish spaces (through the paper, we will only be interested in the case $X=Y=\R^d$), and let $\mu$ and $\nu$ be two probabilities on $X$ and $Y$ respectively. A probability measure $\gamma$ on $X\times Y$ is called a \emph{transport plan} if its two marginals on $X$ and $Y$ coincide with $\mu$ and $\nu$; the collection of the transport plans is denoted by $\Pi(\mu,\nu)$. Given a l.s.c. function $c:X\times Y\to [0,+\infty]$, the \emph{cost} of the plan $\gamma$ is given by
\[
\iint_{X\times Y} c(x,y)\, d\gamma(x,y)\,,
\]
and $\gamma$ is called an \emph{optimal transport plan} if it minimizes the cost among elements of $\Pi(\mu,\nu)$. A particular case of cost functions, namely, the relativistic ones, is now introduced.

\begin{definition}\label{defrelco}
Let $\BigC$ be a closed, bounded, convex set in $\R^d$, containing the origin in its interior, and let $h:\R^d\to [0,+\infty]$ be a function which is strictly convex and bounded on $\BigC$, constantly $+\infty$ on $\R^d\setminus \BigC$, and such that $h(0)=0$. Then, for any $t>0$ we define the function $c_t:\R^d\times \R^d\to [0,+\infty]$ as
\[
c_t(x,y) = h\bigg(\frac {y-x}t\bigg)\,.
\]
Such functions $c_t$ are called \emph{relativistic cost functions}, and for any $t>0$ we denote by $\C_t(\gamma)$ the cost of a plan $\gamma\in\Pi(\mu,\nu)$ with respect to $c_t$, and by $\C(t)$ the minimum of these costs. For simplicity of notations, we will denote by $\|h\|_{L^\infty}$ the maximum of $h$ on $\BigC$.
\end{definition}

A useful concept in mass transportation is the composition of plans, which can be simply defined thanks to a disintegration of the plans (for the definition of disintegration of measures, one can refer for instance to~\cite{AFP}).

\begin{definition}[Composition of transport plans]\label{compos}
Let $\mu,\, \alpha$ and $\nu$ be three probability measures on the Polish spaces $X,\, Y$ and $Z$ respectively, and $\gamma_1\in \Pi(\mu,\alpha)$ and $\gamma_2\in \Pi(\alpha,\nu)$ be two transport plans. Let us disintegrate $\gamma_1$ and $\gamma_2$ with respect to the projection on $Y$, so writing $\gamma_1= \gamma_1^y\otimes \alpha$ and $\gamma_2 = \alpha \otimes \gamma_2^y$. Then, the \emph{composition of $\gamma_1$ and $\gamma_2$} is the transport plan $\gamma_2\circ\gamma_1\in \Pi(\mu,\nu)$, defined as
\[
\iint_{X\times Z} \varphi(x,z) \, d\gamma_2\circ\gamma_1(x,z) := \int_Y \bigg( \iint_{X\times Z} \varphi(x,z)\, d\gamma_1^y(x)\, d\gamma_2^y)^z\bigg) \, d\alpha (y)
\]
for every $\varphi\in {\rm C}_b(X\times Z)$. It is immediate to check that $\gamma_2\circ \gamma_1$ is a transport plan with marginals $\mu$ and $\nu$, as well as that $(x,z)\in\spt(\gamma_2\circ\gamma_1)$ if and only if there exists some $y\in Y$ such that $(x,y)\in\spt\gamma_1$ and $(y,z)\in \spt\gamma_2$.
\end{definition}

We present here two elementary results about convex sets and convex functions.
\begin{lemma}\label{stupgeolem}
Let $c_t$ be a relativistic cost function, let $A$ be the minimum of $|v|$ for vectors $v\in\partial\BigC$, and let $\bar\delta$ and $r^+$ be positive numbers, both much smaller than $A$. Then, there exists an increasing function $\phi:\R^+\to\R^+$, with $\lim_{\eta\to 0} \phi(\eta)=0$, such that
\[
\bigg|h(v)-h\bigg(\frac{v-\eta w}t\bigg) \bigg| < \phi(\eta)
\]
holds for every $v,\, w\in \R^d$ and $t,\, \eta>0$ satisfying
\begin{align*}
v \in \BigC\setminus(1-\bar\delta)\BigC\,, && |w-v| \leq r^+\sqrt d\,, && |1-t| < \frac \eta 2\,.
\end{align*}
\end{lemma}
\begin{proof}
Since the function $h$ is uniformly continuous on $\BigC$, we only have to show that, for $v,\, w$ and $t$ as in the claim, one has
\[
\frac{v-\eta w}t \in \BigC\,,
\]
and this is in turn a trivial geometric property, since $\BigC$ is a bounded, convex set.
\end{proof}

Let us now give the following definition of ``directional derivative'' for a relativistic cost.

\begin{definition}\label{dirder}
Let $\BigC$ and $h$ be as in Definition~\ref{defrelco}, let $P\in\BigC$, and let $v\in\S^{d-1}$ be an \emph{internal direction}, that is, a vector such that $P+\eps v$ belongs to the interior of $\BigC$ for every $0<\eps\ll 1$. We define then \emph{directional derivative} of $h$ at $P$ in the direction $v$ the number
\[
D_v h(P) := \lim_{\eps\searrow 0} \frac{h(P+\eps v)-h(P)}\eps\,.
\]
Notice that $D_v h(P)$ is the right derivative at $0$ of a real one-dimensional convex function defined in a right neighborhood of $0$, hence it always exists and it belongs to $[-\infty,+\infty)$. Notice also that the set $\S_P$ of the internal directions at $P$ is the whole $\S^{d-1}$ if $P$ belongs to the interior of $\BigC$, while for $P\in\partial\BigC$ it is an open subset of $\S^{d-1}$, in particular an open half-sphere if $\partial\BigC$ admits a normal vector at $P$.
\end{definition}

We can now observe that the directional derivatives are either all finite or all infinite.
\begin{lemma}\label{alwrealwin}
Let $\BigC$ and $h$ be as in Definition~\ref{defrelco}, and $P$ and $\S_P$ as in the Definition~\ref{dirder}. Then, the function $v\mapsto D_v h(P)\in [-\infty,+\infty)$, defined on $\S_P$, is continuous. In addition, it is either real on the whole $\S_P$, or constantly $-\infty$.
\end{lemma}
\begin{proof}
The only interesting case is when $P\in\partial\BigC$, since otherwise everything is trivial by the convexity of $h$ and $\BigC$, and in particular $D_v h(P)$ is always real. Hence, we suppose from now on that $P\in\partial\BigC$; let us first prove that the map $v\mapsto D_v$ is either always real or constantly $-\infty$, and then the continuity.
\step{I}{The realness (or constant infiniteness) of the directional derivatives.}
Let us start by taking some internal direction $v\in \S_P$, and let us assume for a moment that $D_v h(P)>-\infty$, so $D_v h(P)\in\R$. Let now $\omega$ be another internal direction; an immediate geometrical consideration ensures that, if $\omega$ is close enough to $v$, then there exists some positive constant $\ell>0$ with the following property. For every $0<\delta\ll 1$, if we define
\begin{align}\label{defQR}
Q = P +\delta v\,, && R = P + \frac \delta 2\, \omega\,, && d = |Q - R| \leq \frac 32 \, \delta
\end{align}
and we let $S$ be the point on the half-line starting at $R$ and passing through $Q$ having distance $d+\ell$ from $R$, then $S\in\BigC$. Notice carefully that a possible value of $\ell$, as well as of the necessary closeness between $\omega$ and $v$, can be obtained independently on $v$, and only depending on the exact form of $\BigC$, as well as on the distance between $v$ and the boundary of $\S_P$.\par
We can now write $Q$ as a convex combination of $R$ and $S$, in fact by construction one has
\begin{equation}\label{convcomb}
Q = \frac \ell{\ell + d} \, R+\frac d{\ell + d} \,S\,.
\end{equation}
As a consequence, the convexity of $h$ ensure
\[
h(P) + \delta D_v h(P) \leq h(Q)
\leq \frac\ell{\ell + d} \, h(R)+\frac d{\ell + d} \,h(S)
\leq h(R)+\frac d{\ell + d} \, \| h\|_{L^\infty}
\leq h(R)+\frac {3\delta}{2\ell} \, \| h\|_{L^\infty}\,,
\]
from which we get
\[
h\bigg(P+\frac \delta 2 \omega\bigg) - h(P) \geq  \delta \bigg( D_v h(P) - \frac 3{2\ell}\, \|h\|_{L^\infty}\bigg)
\]
and hence, sending $\delta\searrow 0$,
\begin{equation}\label{notcont}
D_\omega h(P) \geq 2 D_v h(P) -\frac 3\ell\, \|h\|_{L^\infty}\,.
\end{equation}
This estimate immediately ensures that the subset of $\S_P$ made by the directions along which the directional derivative is not $-\infty$ form an open subset of $\S_P$. But actually, since, as underlined above, the amplitude of the neighborhood of $v$ in which the estimate~(\ref{notcont}) holds only depends on the distance of $v$ from the boundary of $\S_P$, then we directly get that, if at some direction $v\in\S_P$ one has $D_v h(P)>-\infty$, then the same inequality holds for all the directions $v\in \S_P$.\par

Summarizing, we have proved that $D_v h(P)$ is never $+\infty$, and actually it is either real for every $v\in \S_P$, or equal to $-\infty$ for every $v\in \S_P$. This concludes the first step of the proof.

\step{II}{The continuity.}
In this step we prove the continuity, which does not come from~(\ref{notcont}). Thanks to Step~I, we only have to consider the case when $D_v h(P)$ is real for every $v\in\S_P$.\par
Let us then take a direction $v\in\S_P$, select some $0< \eps\ll 1$, and take some $\ell>0$ such that $V=P+\ell v$ belongs to the interior of $\BigC$ and satisfies $h(V) \leq h(P) + \ell (D_v h(P)+ \eps)$. Let now $\omega\in\S_P$ be another direction, and for $\delta\ll \ell$ define again $Q$ and $R$ as in~(\ref{defQR}), and $S$ as in the following line. The point $S$ is arbitrarily close to $V$ (thus also in the interior of $\BigC$) as soon as $\delta$ is small enough and $\omega$ is close enough to $v$, hence by the continuity of $h$ we can assume
\[
h(S) \leq h(P) + \ell (D_v h(P) + 2\eps)\,.
\]
Notice that, differently with what happened in Step~I, this time the value of $\ell$ and the needed closeness of $\omega$ to $v$ really depend on $v$ (this is why the argument of the present step could not prove the realness that, instead, we got in Step~I). Formula~(\ref{convcomb}) is still valid, with $d=|Q-R|$ being this time arbitrarily close to $\delta/2$, again up to select $\omega$ close enough to $v$. Hence we can evaluate, similarly with what we have done in Step~I,
\[\begin{split}
h(P) + \delta D_v h(P) &\leq h(Q)
\leq \frac\ell{\ell + d} \, h(R)+\frac d{\ell + d} \,h(S)\\
&\leq \bigg( 1 - \frac d {\ell + d}\bigg) h\bigg( P + \frac \delta 2 \,\omega\bigg) + \frac d {\ell + d}\, h(P) + \frac {d\ell}{\ell+d}\, (D_v h(P)+2\eps)\,,
\end{split}\]
from which it readily follows
\begin{equation}\label{esti1}
\lim_{\delta \searrow 0}\, \frac{h\big( P +\frac \delta 2\, \omega \big) - h(P)}{\frac \delta 2} \geq D_v h(P) - 2\eps\,.
\end{equation}
The opposite inequality, namely,
\begin{equation}\label{esti2}
\lim_{\delta \searrow 0}\, \frac{h\big( P +\frac \delta 2\, \omega \big) - h(P)}{\frac \delta 2} \leq D_v h(P) + 2\eps\,,
\end{equation}
can be obtained in the very same way, just exchanging the role of $Q$ and $R$; more precisely, we define again $Q= P +\delta v$, but this time $R= P + 2\delta \omega$, and this time $S$ has distance $\ell$ from $R$, on the half-line starting at $Q$ and passing through $R$. Hence, this time we can write $R$ as a convex combination of $Q$ and $S$, and then the very same calculation which brought to~(\ref{esti1}) now bring to~(\ref{esti2}). Since $\eps$ is arbitrary, the continuity is then established and the proof is concluded.
\end{proof}

%

Thanks to the above result, it is now very simple to introduce the ``highly relativistic'' cost functions.

\begin{definition}\label{defhighrelco}
Let $h$ be as in Definition~\ref{defrelco}. The functions $c_t$ are called \emph{highly relativistic cost functions} if $D_v h(P)$ is constantly $-\infty$ on $\S_P$ for every $P\in \partial\BigC$.
\end{definition}

To conclude this section, we give the claim of the Chain Lemma (for a proof, see~\cite[Proposition~2.11]{bpp}).

\begin{lemma}[Chain Lemma]\label{chle}
Let $\gamma,\,\gamma'\in \Pi(\mu,\nu)$, $0\neq \gamma_0\leq \gamma$, and set $\mu_0=\pi_1\gamma_0$, $\nu_0=\pi_2\gamma_0$. Then, there exist $N\in\N$ and $\bar\eps>0$ such that, for every $\eps\leq\bar\eps$, there are plans $\tilde\gamma\leq \gamma$ and $\tilde\gamma'\leq \gamma'$ satisfying
\begin{align*}
&\pi_1\tilde\gamma=\pi_1\tilde\gamma'=\tilde\mu+\mu_A\,, &&\pi_2 \tilde\gamma=\tilde\nu+\nu_A\,, && \pi_2 \tilde\gamma'=\tilde\nu+\nu_B\,,\\
&\mu_A \leq \mu_0\,,\qquad \tilde\mu\leq \mu-\mu_0\,, && \nu_A,\, \nu_B \leq \nu_0\,, && \tilde\nu \wedge \nu_0 = 0\,,\\
&\|\tilde\gamma\|=\|\tilde\gamma'\|=(N+1)\eps\,, && \|\tilde\mu\|=\|\tilde\nu\|=N\eps\,, && \| \mu_A\|= \|\nu_A\|=\| \nu_B\| =\eps\,. 
\end{align*}
In particular, $\tilde\gamma$ can be decomposed as $\tilde\gamma= \tilde\gamma_0 + \tilde\gamma_\infty$, where
\begin{align*}
\tilde\gamma_0 \leq \gamma_0\,, &&
\tilde\gamma_\infty\leq \gamma - \gamma_0\,, &&
\pi_1\tilde\gamma_0=\mu_A\,, && \pi_2\tilde\gamma_0=\nu_A\,, &&
\pi_1\tilde\gamma_\infty=\tilde\mu\,, && \pi_2\tilde\gamma_\infty=\tilde\nu\,.
\end{align*}
\end{lemma}

\section{Continuity of the total cost\label{seccont}}

This section is devoted to prove the continuity of the total cost, that is, part~(i) of Theorem~\mref{Main}.

\proofof{Theorem~\mref{Main}, part~(i)}
Since the right continuity of the function $t\mapsto \C(t)$ is obvious (and it was also proved in the literature, see Theorem~\ref{first}), we only have to deal with the left continuity. Up to rescaling, we can assume that $T<1$ and we aim to prove the left continuity at $t=1$. Let us fix any $T<t'<1$, and let us call $\gamma$ and $\gamma'$ the optimal transport plans corresponding to $t=1$ and $t=t'$. Moreover, let us define
\[
A:= \min \big\{ |v|:\, v\in\partial\BigC\big\}\,,
\]
let us fix two small constants $r^+\ll \bar\delta\ll 1-t'$ such that
\begin{equation}\label{rbdpic}
r^+\sqrt d \ll \bar\delta A \,,
\end{equation}
and let $r$ be a third positive constant, much smaller than $r^+$. Define then the set
\[
S_{\bar\delta} := \Big\{ (x,y)\in\R^d\times\R^d:\, y-x \in (1-\bar\delta) \,\BigC \Big\}
\]
and the measure
\[
\gamma_0 := \gamma \res \Big( \R^d\times \R^d \setminus S_{\bar \delta}\Big)\,.
\]
Notice that $\gamma_0$ is the part of the transport made by the points which ``do not move too much'', that is, $y-x$ is not too close to the boundary of $\BigC$; as a consequence, it is immediate to observe that $\C_t(\gamma_0)$ converges to $\C_1(\gamma_0)$ for $t\nearrow 1$. Instead, $\gamma-\gamma_0$ is the part of the transport with which we have to deal carefully. Notice also that we do not even know that $\|\gamma-\gamma_0\|$ is small if $\bar\delta$ is small enough: this would be the case only for highly relativistic transport costs. In order to prove the claim, we will build a transport plan $\xi\in\Pi(\mu,\nu)$ whose cost satisfies
\begin{equation}\label{toprove}
\limsup_{t\nearrow 1} \C_t(\xi) \leq \C_1(\gamma) + 2 \phi\bigg(\frac{2r} A\bigg)\,,
\end{equation}
where $\phi$ is the function given by Lemma~\ref{stupgeolem}. Notice that the function $\phi$ depends on $\bar\delta$ and on $r^+$, but not on $A$ (which is a geometrical constant, only depending on $\BigC$), neither on $r$ (which is an arbitrarily small constant, in turn depending on $\bar\delta$ and $r^+$). Since the function $\phi$ is infinitesimal for $r\searrow 0$, the searched left continuity of $t\mapsto \C(t)$ at $t=1$ will be established once we prove~(\ref{toprove}). Unfortunately, the trivial choice $\xi=\gamma$ does not work, one could even have $\C_t(\gamma)=+\infty$ for every $t<1$. For the sake of clarity, we divide our construction in some steps.

\step{I}{Definition of the cubes $Q_i^1$ and $Q_i^2$, and of the measures $\gamma_i$, $\mu_i$ and $\nu_i$, and $\gamma_i^K$, $\mu_i^K$ and $\nu_i^K$.}
We start by covering the support of $\gamma-\gamma_0$ with finitely many cubes of side $r$ in $\R^d\times\R^d$, and we call $M$ their number. More precisely, for every $1\leq i \leq M$ we select two cubes $Q^1_i$ and $Q^2_i$ of side $r$ in $\R^d$, and a positive measure $\gamma_i$ concentrated in $Q^1_i\times Q^2_i$, in such a way that
\begin{itemize}
\item $\gamma-\gamma_0 = \sum_{i=1}^M \gamma_i$\,;
\item $\gamma_i\wedge \gamma_j=0$ for every $i\neq j$\,;
\item for every $i$, the cubes $Q^1_i$ and $Q^2_i$ are parallel; more precisely, $Q^2_i = Q^1_i + v_i$, and the vector $v_i$ is parallel to one of the sides of $Q^1_i$ (hence, also of $Q^2_i$).
\end{itemize}
Notice that the existence of such cubes and measures is obvious (in fact, we do not require the products $Q^1_i\times Q^2_i$ to be disjoint, but only the measures $\gamma_i$ to be mutually singular). We will call $\mu_i$ and $\nu_i$ the two marginals of $\gamma_i$ and, up to remove useless cubes, we will assume that
\[
m_i := \|\gamma_i\| = \|\mu_i\| = \|\nu_i\| >0
\]
for every $i$. We apply now the Chain Lemma~\ref{chle} to the measures $\gamma,\, \gamma'$ and $\gamma_i$, obtaining some constants $N_i$ and $\bar\eps_i$, and we call $N = \max_i N_i$ and $\bar\eps = \min_i \bar\eps_i$. Then, we fix an arbitrarily small constant $\eps\ll \min_i \{m_i\}\bar\eps/MN$, also satisfying
\begin{equation}\label{defeps}
\eps N \ll \phi\bigg(\frac{2r} A\bigg)\,,
\end{equation}
and we find $K\gg 1$ such that for every $i$ one has
\begin{equation}\label{Kgrande}
\mu_i \Big(\big\{x:\, \rho(x)\geq K\big\} \Big)  \leq \eps m_i\,,
\end{equation}
being $\rho$ the density of $\mu$ with respect to the Lebesgue measure $\L$. Then, we call
\[
\gamma_i^K = \gamma_i \res \Big\{ (x,y)\in\R^d\times \R^d:\, \rho(x) \leq K\Big\}\,,
\]
and we let $\mu_i^K$ and $\nu_i^K$ be its two marginals.

\step{II}{Definition of the transport plan $\xi^1_i$ and its cost.}
In this step, we provide the first part of the ``competitor'' transport plan $\xi$, namely, a transport plan $\xi^1_i$ for every $1\leq i\leq M$. Since this will be done for each $i$ independently, we concentrate ourselves, only within this step, on a given $1\leq i \leq M$. For further clarity, this step will be further subdivided in two substeps. Let us consider the cubes $Q^1_i$ and $Q^2_i$. For simplicity of notations, and up to a rotation and translation, we can assume that $Q^1_i=[0,r]^d$, while $Q^2_i=[a,a+r]\times [0,r]^{d-1}$. Notice that, for the ease of notation, we call ``$a$'', and not ``$a_i$'', the distance between the cubes; more in general, since in this step we only work with a fixed index $i$, we will not put a subscript ``$i$'' to all the quantities which are used only inside this step. Keep in mind that the measure $\gamma_i$ is concentrated in $\R^d\times\R^d \setminus S_{\bar \delta}$, hence by~(\ref{rbdpic}) and the fact that $a+r\sqrt d>(1-\bar\delta) A$, we have $r\sqrt d\ll a$.

\step{IIa}{Definition of the relevant sets and of the plan $\xi^1_i$.}
In this first substep, we give the definition of the plan $\xi^1_i$. Let us use the notation $x=(\sigma,\tau)\in\R\times\R^{d-1}$ for points in $\R^d$, denote by $\pi_\tau:\R^d\to\R^{d-1}$ the projection on the variable $\tau$, and disintegrate the measure $\mu_i^K$ with respect to $\pi_\tau$, obtaining the decomposition
\[
\mu_i^K = \mu_\tau \otimes \alpha\,,
\]
where $\alpha=(\pi_\tau)_\# \mu_i^K$ and the measure $\mu_\tau$ is a probability measure concentrated in $[0,r]$ for $\alpha$-a.e. $\tau$. Let us now fix an arbitrarily small constant $\delta$, much smaller than $\eps$, and for $\alpha$-a.e. $\tau\in [0,r]^{d-1}$ let
\[
f_\tau: \Big\{ (\sigma,\tau)\in [0,r]^d:\, \mu_\tau\big([0,\sigma]\big) \leq 1-\delta\Big\} \to \Big\{ (\sigma,\tau)\in [0,r]^d:\, \mu_\tau\big([0,\sigma]\big) \geq \delta\Big\}
\]
be the measurable function given by
\[
\mu_\tau \big((\sigma,f_\tau(\sigma)]\big) = \delta\,.
\]
Notice that these functions are well-defined because $\mu^i_K$ is absolutely continuous with respect to the $d$-dimensional Lebesgue measure, and thus for $\alpha$-a.e. $\tau$ the measure $\mu_\tau$ is absolutely continuous with respect to the $1$-dimensional Lebesgue measure. If the density of $\mu^K_i$ is constant, then the functions $f_\tau$ are nothing else than the right translation of a distance $r\delta$; therefore, we can expect that most of the functions $f_\tau$ move points to the right more or less of a distance comparable with $r\delta$. More precisely, we define a large constant $H$ as
\begin{equation}\label{defofH}
H = \frac{K r^d}{\eps m_i}\,;
\end{equation}
notice that $H$ depends on $r$, on $\eps$, on the measures $\gamma_i$ (thus on $\bar\delta$) and on $K$ (so, again on $\eps$ and on $\bar\delta$), but not on $\delta$: in particular, $\delta/H$ is arbitrarily small. We define now
\[
Z = \bigg\{ \tau\in [0,r]^{d-1}:\, \exists\, \sigma,\, f_\tau(\sigma) - \sigma <\frac{r\delta}H \bigg\}\,,
\]
and we claim that
\begin{equation}\label{alphaZsmall}
\alpha(Z) \leq \eps m_i\,.
\end{equation}
Indeed, by the Measurable Selection Theorem we can select a measurable function $\tau\mapsto\sigma(\tau)$, which associates to every $\tau\in Z$ some $\sigma(\tau)$ with the property that $f_\tau(\sigma(\tau))-\sigma(\tau)<r\delta/H$, and we can also define the ``box''
\[
\Gamma = \bigg\{ (\sigma,\tau)\in [0,r]^d:\, \tau\in Z,\, \sigma(\tau)<\sigma< \sigma(\tau) + \frac{r\delta}H \bigg\}\,.
\]
Thus, by Fubini Theorem and recalling the decomposition $\mu^K_i = \mu_\tau\otimes \alpha$, on one side we have that
\[
\mu^K_i(\Gamma) \leq K \L(\Gamma) \leq K\, \frac{r\delta}H \H^{d-1}(Z)\leq K\, \frac{r^d\delta}H\,,
\]
and on the other side that
\[
\mu^K_i(\Gamma) = 
\int_{\tau\in Z} \mu_\tau \bigg(\Big(\sigma(\tau), \sigma(\tau)+\frac{r\delta}H\,\Big)\bigg)\, d\alpha(\tau)
\geq\int_{\tau\in Z} \mu_\tau \Big(\big(\sigma(\tau), f_\tau(\sigma(\tau))\big)\Big)\, d\alpha(\tau)
=\delta \alpha(Z)\,,
\]
thus by the choice~(\ref{defofH}) the estimate~(\ref{alphaZsmall}) follows.\par
We can now go into the definition of the plan $\xi^1_i$; the very rough idea is to ``copy'' the original plan $\gamma_i$, but instead of sending a generic point $x=(\sigma,\tau)$ onto $y$, we send the corresponding point $(f_\tau(\sigma),\tau)$ onto $y$. If $\tau\notin Z$, then we are sure that $f_\tau(\sigma)$ is at least a given bit more on the right, with respect to $\sigma$, and then the distance has been decreased and it is reasonable to hope that the cost has been lowered. Let us now make it formal: we define the sets
\begin{align*}
L=\Big\{ (\sigma,\tau)\in [0,r]^d:\, \tau\notin Z,\,\mu_\tau\big([0,\sigma]\big) \leq 1-\delta\Big\}\,, &&
R=\Big\{ (\sigma,\tau)\in [0,r]^d:\, \tau\notin Z,\,\mu_\tau\big([0,\sigma]\big) \geq \delta\Big\}\,,
\end{align*}
the functions $\tilde g: L\to [0,r]^d$ and $g: L\times \R^d \to \R^d\times \R^d$ as
\begin{align*}
\tilde g(\sigma,\tau) = \big(f_\tau(\sigma),\tau\big)\,, && g(x,y) = \big(\tilde g(x),y\big)\,,
\end{align*}
and then the plan $\xi^1_i$ as
\[
\xi^1_i = g_\# \Big(\gamma^K_i \res L\times \R^d \Big)\,.
\]
It is not difficult to check that the two marginals of $\xi^1_i$ are given by
\begin{align*}
\pi_1 \xi^1_i (A)= \mu^K_i ( A \cap R)\leq \mu_i(A)\,, &&
\pi_2 \xi^1_i (B)= \gamma^K_i ( L\times B)\leq \nu_i(B) \,:
\end{align*}
here and in the following, for simplicity of notations, we denote the two marginals of a generic plan $\zeta$ as $\pi_1 \zeta$ and $\pi_2 \zeta$, instead than ${\pi_1}_\# \zeta$ and ${\pi_2}_\# \zeta$. As a consequence, if we write
\begin{align*}
\mu_i = \pi_1 \xi^1_i+ \mu_{i,{\rm rem}}\,, && \nu_i = \pi_2 \xi^1_i + \nu_{i,{\rm rem}}\,,
\end{align*}
we can evaluate by~(\ref{Kgrande}) and~(\ref{alphaZsmall})
\begin{equation}\label{whatremains}
\|\nu_{i,{\rm rem}}\| = \|\mu_{i,{\rm rem}}\| 
= \|\mu_i - \mu_i^K\| + \mu_i^K(Q^1_i\setminus R)
= \|\mu_i - \mu_i^K\| + \delta \alpha([0,r]^{d-1}\setminus Z) + \alpha(Z)
\leq \big(2\eps +\delta\big) m_i\,.
\end{equation}

\step{IIb}{Estimate on the cost of $\xi^1_i$.}
In this substep, we obtain an estimate on the cost of the plan $\xi^1_i$. Let us take $(z,y)\in\spt\xi^1_i$: this means that $z=\tilde g(x)$ for some $x\in L$, with $(x,y)\in\spt \gamma^K_i$; in particular, $y-x\in S_{\bar\delta}$, and if we write $x=(\sigma,\tau)$, then $z=(f_\tau (\sigma),\tau)$. Hence, we have that $z-x=\big(f_\tau(\sigma)-\sigma\big){\rm e}_1$; let us then call $v=y-x$, $w=a{\rm e}_1$ and $\eta = \big(f_\tau(\sigma)-\sigma\big)/a$, so that $y-z=v - \eta w$. Since $(x,y)\in\spt \gamma^K_i$, we have that $v=y-x\in \BigC\setminus (1-\bar\delta)\BigC$, and by construction we have that $|w-v|\leq r\sqrt d \ll r^+ \sqrt d$. Moreover, $x\in L$, thus
\[
\frac {r\delta}{aH} \leq \eta\leq \frac ra < \frac {2r}A\,.
\]
As a consequence, if
\[
1-\frac{r\delta}{2H} <t<1\,,
\]
then for sure $|1-t|<\eta/2$, hence we can apply Lemma~\ref{stupgeolem} to find that
\[\begin{split}
\big|c_1(x,y)-c_t\big(g(x,y)\big)\big|
&=\big|c_1(x,y)-c_t(z,y)\big|
=\bigg|h(y-x)-h\bigg(\frac{y-z}t\bigg) \bigg|
=\bigg|h(v)-h\bigg(\frac{v-\eta w}t\bigg) \bigg| < \phi(\eta)\\
&< \phi\bigg(\frac {2r}A\bigg)\,.
\end{split}\]
By the definition of $\xi^1_i$, we get then
\begin{equation}\label{cost1}\begin{split}
\limsup_{t\nearrow 1} \C_t(\xi^1_i) 
&=\limsup_{t\nearrow 1} \iint_{\R^d\times\R^d} c_t(z,y)\, d\xi^1_i(z,y)
=\limsup_{t\nearrow 1} \iint_{\R^d\times\R^d} c_t\big(g(x,y)\big)\, d\gamma^K_i(x,y)\\
&\leq \C_1(\gamma^K_i) + \| \gamma^K_i\| \,\phi\bigg(\frac {2r}A\bigg)\,.
\end{split}\end{equation}

\step{III}{Definition of the tranport plans $\xi^2$ and $\xi^3$.}
In the preceding step, we have found a transport plan $\xi^1_i$ which is sending ``almost all'' of $\mu_i$ onto ``almost all'' of $\nu_i$. To complete the construction of our competitor transport plan $\xi$, we have then to fix the remaining parts of the $\mu_i$ and $\nu_i$, as well as to send $\mu_0$ onto $\nu_0$. To do so, we will make use of the Chain Lemma~\ref{chle}. More precisely, for every $i$ we apply the Chain Lemma with constant $M\|\mu_{i,{\rm rem}}\|$: notice that this is possible only if the constant is smaller than $\eps_i$, but in fact by~(\ref{whatremains})
\[
M\|\mu_{i,{\rm rem}}\| \leq M (2\eps+\delta)m_i \leq 3\eps M \ll \bar\eps \leq \bar\eps_i\,.
\]
The Chain Lemma then provides us with measures, which we call for simplicity $M\tilde\gamma_i$ and $M\tilde\gamma_i'$; then, $\tilde\gamma_i=\tilde\gamma_i^A+\tilde\gamma_i^\infty$ with
\begin{align}\label{norms}
M\tilde\gamma_i^A \leq \gamma_i\,, && M \tilde\gamma_i^\infty \leq \gamma-\gamma_i \,, &&
\|\tilde\gamma_i^A\| = \|\mu_{i,{\rm rem}}\|\,, &&
\|\tilde\gamma_i^\infty\| = N_i \|\mu_{i,{\rm rem}}\|\leq N \|\mu_{i,{\rm rem}}\|\,.
\end{align}
Notice that we have the inequality $\tilde\gamma_i^\infty\leq \gamma-\gamma_i$, but this does not mean $\tilde\gamma_i^\infty\leq\gamma_0$, since $\tilde\gamma_i^\infty$ might have parts in common with $\gamma_j$ for some $j\neq i$. We can then further subdivide $\tilde\gamma_i^\infty= \tilde\gamma_i^{\rm OUT}+\tilde\gamma_i^{\rm NO}$, with
\begin{align}\label{projOUT}
M\tilde\gamma_i^{\rm OUT}\leq \gamma_0\,, && M \tilde\gamma_i^{\rm NO} \leq \gamma-\gamma_0-\gamma_i\,.
\end{align}
The marginals of these measures are
\begin{align*}
\pi_1 \tilde\gamma_i^A = \mu_i^A\leq \mu_i\,, &&
M\pi_1 \tilde\gamma_i^{\rm OUT} \leq \mu_0\,, &&
\pi_1 \tilde\gamma_i^{\rm NO} \leq \mu-\mu_0-\mu_i\,, \\
\pi_2 \tilde\gamma_i^A = \nu_i^A\leq \nu_i\,, &&
M\pi_2 \tilde\gamma_i^{\rm OUT} \leq \nu_0\,, &&
\pi_2 \tilde\gamma_i^{\rm NO} \leq \nu-\nu_0-\nu_i\,.
\end{align*}
Instead, the marginals of $\tilde\gamma_i'$ are given by
\begin{align}\label{projtilgam'}
\pi_1 \tilde\gamma_i' = \pi_1 \tilde\gamma_i = \mu_i^A + \pi_1(\tilde\gamma_i^{\rm OUT} + \tilde\gamma_i^{\rm NO})\,, &&
\pi_2 \tilde\gamma_i' = \nu_i^B + \pi_2(\tilde\gamma_i^{\rm OUT} + \tilde\gamma_i^{\rm NO})\,,
\end{align}
where $\nu_i^B$ does not necessarily coincide with $\nu_i^A$, but they are both measures of norm $\|\mu_{i,{\rm rem}}\|$ smaller than $\nu_i$. We are then ready to define the plan $\xi^2$ as
\[
\xi^2 = \gamma_0 - \gamma_0^{\rm OUT}\,, \qquad \hbox{where} \qquad
\gamma_0^{\rm OUT}:= \sum_{i=1}^M \tilde\gamma_i^{\rm OUT}\,.
\]
Notice that $\xi^2$ is a positive measure because, according to~(\ref{projOUT}), we do not have just $\tilde\gamma_i^{\rm OUT}\leq \gamma_0$, but also $\tilde\gamma_i^{\rm OUT}\leq \gamma_0/M$: in fact, the reason why we have applied the Chain Lemma with constants $M\|\mu_{i,{\rm rem}}\|$ instead of just $\|\mu_{i,{\rm rem}}\|$ was precisely to be sure to get, at this point, a \emph{positive} measure $\xi^2$. We also call
\begin{align*}
\mu_0^{\rm OUT} = \pi_1 \gamma_0^{\rm OUT} \leq \mu_0 \,, &&
\nu_0^{\rm OUT} = \pi_2 \gamma_0^{\rm OUT} \leq \nu_0 \,,
\end{align*}
so that the marginals of $\xi^2$ are
\begin{align}\label{projxi2}
\pi_1 \xi^2 = \mu_0-\mu_0^{\rm OUT}\,, &&
\pi_2 \xi^2 = \nu_0-\nu_0^{\rm OUT}\,.
\end{align}
Let us now set $\gamma^{\rm NO} = \sum_{i=1}^M \tilde\gamma_i^{\rm NO}$; since again by~(\ref{projOUT}) $\gamma^{\rm NO}\leq \gamma-\gamma_0$, we can decompose it as
\[
\gamma^{\rm NO}=\sum_{i=1}^M \gamma^{\rm NO}_i\,, \qquad\hbox{with}\qquad \gamma^{\rm NO}_i\leq \gamma_i\,.
\]
Notice that $\gamma^{\rm NO}_i$ does not coincide with $\tilde\gamma_i^{\rm NO}$; on the contrary, $\gamma^{\rm NO}_i$ comes from the measures $\tilde\gamma_j^{\rm NO}$ for all $j\neq i$. Let us finally call $\mu_i^{\rm NO}$ and $\nu_i^{\rm NO}$ the two marginals of $\gamma_i^{\rm NO}$, and notice that by~(\ref{norms}) and~(\ref{whatremains})
\[
\|\mu_i^{\rm NO}\|= \|\gamma_i^{\rm NO}\| \leq \|\gamma^{\rm NO}\|
= \sum_{i=1}^M \|\tilde\gamma_i^{\rm NO}\|
\leq \sum_{i=1}^M \|\tilde\gamma_i^\infty\|
\leq \sum_{i=1}^M N \| \mu_{i,{\rm rem}}\|
\leq N \sum_{i=1}^M 3\eps m_i
\leq 3\eps N\ll m_i\,,
\]
while
\[
\|\xi^1_i\|= \|\mu_i - \mu_{i,{\rm rem}}\|  \geq (1-3\eps) m_i\,.
\]
As a consequence, all the constants $\lambda_i$ given by
\[
\lambda_i = 1 - \frac{\|\mu_i^{\rm NO}\|}{\|\xi^1_i\|}
\]
are only slightly smaller than $1$; define then
\begin{align*}
\mu_{i,{\rm rem}}^+ = \mu_i - \pi_1 \big( \lambda_i \xi^1_i\big)\geq \mu_{i,{\rm rem}}\,, &&
\nu_{i,{\rm rem}}^+ = \nu_i - \pi_2 \big( \lambda_i \xi^1_i\big)\geq \nu_{i,{\rm rem}}\,,
\end{align*}
and observe that, by the definition of $\lambda_i$, we have
\begin{equation}\label{estiafterlambda}
\|\mu_{i,{\rm rem}}^+\| =\|\nu_{i,{\rm rem}}^+\|
=\|\mu_{i,{\rm rem}}\|+(1-\lambda_i) \|\xi^1_i\|
=\|\mu_{i,{\rm rem}}\|+\|\mu_i^{\rm NO}\|
=\|\mu_i^A\|+\|\mu_i^{\rm NO}\|\,.
\end{equation}
Finally, we can set
\[
\xi^1 := \sum_{i=1}^M \lambda_i \xi^1_i\,,
\]
whose marginals are
\begin{align}\label{projxi1}
\pi_1 \xi^1 = \sum_{i=1}^M \mu_i - \mu_{i,{\rm rem}}^+\,, &&
\pi_2 \xi^1 = \sum_{i=1}^M \nu_i - \nu_{i,{\rm rem}}^+\,.
\end{align}
We aim to define our competitor plan as $\xi=\xi^1+\xi^2+\xi^3$: by~(\ref{projxi2}) and~(\ref{projxi1}), the marginals of $\xi^3$ must satisfy
\begin{align}\label{requixi3}
\pi_1 \xi^3 = \mu_0^{\rm OUT}+\sum_{i=1}^M \mu_{i,{\rm rem}}^+\,, &&
\pi_2 \xi^3 =\nu_0^{\rm OUT}+\sum_{i=1}^M \nu_{i,{\rm rem}}^+\,.
\end{align}
Let us consider the measure $\xi^3_{\rm TEMP}=\sum_{i=1}^M \tilde\gamma'_i$: keeping in mind~(\ref{projtilgam'}), we have
\begin{align*}
&\pi_1 \xi^3_{\rm TEMP} = \sum_{i=1}^M \mu_i^A + \pi_1(\tilde\gamma_i^{\rm OUT} + \tilde\gamma_i^{\rm NO})
= \mu_0^{\rm OUT} + \sum_{i=1}^M \mu_i^A + \mu_i^{\rm NO}\,, \\
&\pi_2 \xi^3_{\rm TEMP} = \sum_{i=1}^M \nu_i^B + \pi_2(\tilde\gamma_i^{\rm OUT} + \tilde\gamma_i^{\rm NO})
= \nu_0^{\rm OUT} + \sum_{i=1}^M \nu_i^B + \nu_i^{\rm NO}\,. 
\end{align*}
Notice that the marginals of $\xi^3_{\rm TEMP}$ are almost exactly those required for $\xi^3$ in~(\ref{requixi3}), the only ``mistake'' being that for each $1\leq i \leq M$ in place of the measures $\mu_{i,{\rm rem}}^+\leq \mu_i$ and $\nu_{i,{\rm rem}}^+\leq \nu_i$ one has the measures $\mu_i^A+\mu_i^{\rm NO}\leq \mu_i$ and $\nu_i^B+\nu_i^{\rm NO}\leq \nu_i$, which have anyway the same mass thanks to the estimate~(\ref{estiafterlambda}). It is then easy to adjust the measure $\xi^3_{\rm TEMP}$: we define the two auxiliary transport plans
\begin{align*}
\beta_1 = ({\rm Id},{\rm Id})_\# \mu_0^{\rm OUT} + \sum_{i=1}^M \mu_{i,{\rm rem}}^+ \otimes \big( \mu_i^A+\mu_i^{\rm NO} \big)\,, &&
\beta_2 = ({\rm Id},{\rm Id})_\# \nu_0^{\rm OUT} + \sum_{i=1}^M  \big( \nu_i^B+\nu_i^{\rm NO} \big)\otimes\nu_{i,{\rm rem}}^+ \,.
\end{align*}
Notice that $\beta_1$ has first marginal $\mu_0^{\rm OUT} + \sum_i \mu_{i,{\rm rem}}^+$ and second marginal $\mu_0^{\rm OUT}+\sum_i \big( \mu_i^A+\mu_i^{\rm NO} \big)$, while $\beta_2$ has first marginal $\nu_0^{\rm OUT}+\sum_i \big( \nu_i^B+\nu_i^{\rm NO} \big)$ and second marginal $\nu_0^{\rm OUT} + \sum_i \nu_{i,{\rm rem}}^+$. Therefore, if we finally define the composition $\xi^3 = \beta_2 \circ \xi^3_{\rm TEMP}\circ \beta_1$ in the sense of Definition~\ref{compos}, then $\xi_3$ is a positive measure whose marginals satisfy~(\ref{requixi3}), thus the plan $\xi=\xi^1+\xi^2+\xi^3$ is an admissible transport plan.

\step{IV}{Estimate on the cost of the transport plan $\xi$.}
In this last step we want to estimate the cost of the transport plan $\xi$: in particular, we will establish~(\ref{toprove}), so concluding the proof. By linearity of the cost, we have of course $\C_t(\xi)=\C_t(\xi^1)+\C_t(\xi^2)+\C_t(\xi^3)$, hence we will consider the three terms separately. Concerning $\xi^1$, it is enough to recall~(\ref{cost1}) and the fact that the constants $\lambda_i$ from Step~III are smaller than $1$, so to get
\begin{equation}\label{estixi1}\begin{split}
\limsup_{t\nearrow 1}\C_t(\xi^1) 
&=\limsup_{t\nearrow 1} \sum_{i=1}^M \lambda_i\C_t( \xi^1_i) 
\leq \sum_{i=1}^M \limsup_{t\nearrow 1} \C_t(\xi^1_i)
\leq \sum_{i=1}^M \C_1(\gamma^K_i) + \| \gamma^K_i\| \,\phi\bigg(\frac {2r}A\bigg)\\
&\leq \C_1(\gamma-\gamma_0) + \phi\bigg(\frac {2r}A\bigg)\,.
\end{split}\end{equation}
Let us now consider $\xi^2$: since $\xi^2\leq \gamma_0$, then of course $\C_t(\xi^2)\leq \C_t(\gamma_0)$. On the other hand, the transport plan $\gamma_0$ is concentrated by definition in $S_{\bar\delta}$; that is, for $\gamma_0$-a.e. $(x,y)$, one has $y-x\in (1-\bar\delta)\BigC$. Since the function $h$ is strictly convex in the whole $\BigC$, it is uniformly Lipschitz in a neighborhood of $(1-\bar\delta)\BigC$, which implies
\[
\sup_{(x,y)\in S_{\bar\delta}} |c_t(x,y) - c_1(x,y)| \freccia{t\to 1} 0 \,.
\]
Consequently, we can simply estimate
\begin{equation}\label{estixi2}
\limsup_{t\nearrow 1}\C_t(\xi^2) \leq \limsup_{t\nearrow 1}\C_t(\gamma_0) \leq \C_1(\gamma_0)\,.
\end{equation}
Finally, let us pass to consider $\xi^3$, which was defined as $\xi^3= \beta_2\circ \xi^3_{\rm TEMP}\circ \beta_1$. Keep in mind that $\xi^3_{\rm TEMP}\leq \gamma'$ by definition, hence a pair $(x,y)\in\spt(\xi^3_{\rm TEMP})$ must satisfy $y-x \in t' \,\BigC\subseteq (1-\bar\delta)\BigC$. Moreover, the plans $\beta_1$ and $\beta_2$ are only moving points inside given squares; more precisely, if $(\tilde x, x)\in\spt\beta_1$ then necessarily $|x-\tilde x| \leq r \sqrt d$, and similarly if $(y,\tilde y)\in\spt\beta_2$ then $|\tilde y -y|\leq r \sqrt d$. As a consequence, keeping in mind~(\ref{rbdpic}) we have that for every $(\tilde x,\tilde y)\in\spt\xi^3$ it is $\tilde y-\tilde x \in (1-\frac {\bar\delta} 2)\,\C$, and then $c_t(\tilde x,\tilde y)< +\infty$ for every $t$ close enough to $1$. Summarizing, for any such $t$ we have, also recalling~(\ref{whatremains}) and~(\ref{defeps}),
\[
\C_t(\xi^3) \leq \|h \|_{L^\infty} \|\xi^3\|
= \|h \|_{L^\infty} \|\xi^3_{\rm TEMP}\| 
\leq (N+1) \sum_{i=1}^M \|\mu_{i,{\rm rem}}\|
\leq 3 \eps (N+1) \sum_{i=1}^M m_i \leq 3 \eps (N+1) \leq \phi\bigg(\frac{2r}A\bigg)\,.
\]
Putting this last estimate together with~(\ref{estixi1}) and~(\ref{estixi2}), we have finally established~(\ref{toprove}), and the proof is complete.
\end{proof}

\section{Directions in the boundary of $\BigC$ with infinite slope are not used by optimal plans\label{seclocat}}

In this section we prove the claim~(ii) of Theorem~\mref{Main}, which says that, for every supercritical time $t>T$, the optimal transport plan $\gamma_t$ does not use vectors in the boundary of $\BigC$ at which the directional derivative of $h$ is $-\infty$. Since this happens at all the vectors in the boundary of $\BigC$ when the cost is highly relativistic (this is indeed the definition of highly relativistic costs), then this claim generalizes the analogous one in Theorem~\ref{Thbpp}. The construction needed to prove this part is very similar to the one that we performed to prove the first part of Theorem~\mref{Main}; in fact, the situation this time is much simpler, because we need to use only a single square in $\R^d\times \R^d$.

\proofof{Theorem~\mref{Main}, part~(ii)}
Up to rescaling, we can assume for simplicity that $T<t=1$. Exactly as in the proof of part~(i), let us arbitrarily fix some $T<t'<1$ and call $\gamma$ and $\gamma'$ the optimal transport plans corresponding to the times $t=1$ and $t=t'$. Recall that we have to show that $\gamma$ does not charge the pairs $(x,y)$ with $y-x\in\Theta$, where the set $\Theta$ is defined as
\[
\Theta = \Big\{ v \in \partial\BigC:\, D_{-v} h(v) =-\infty\Big\}\,.
\]
Assume, instead, that the measure $\gamma\res \{ (x,y):\, y-x\in\Theta\}$ is non trivial, and let $(\bar x,\bar y)$ belong to its support. Up to a rotation and a rescaling, we can assume that $\bar y-\bar x={\rm e}_1$. Let us now fix a small constant $r \ll 1-t'$, call $Q_1$ and $Q_2$ the two squares centered at $\bar x$ and $\bar y$ with sides parallel to the coordinate axes and of length $r$, and define
\[
\gamma_0 = \gamma \res \Big\{ (x,y) \in Q_1\times Q_2:\, y-x \in \Theta\Big\}\,.
\]
Notice that $\gamma_0$ is not the trivial measure by construction; call also $\mu_0$ and $\nu_0$ the marginals of $\gamma_0$; up to a last translation, we assume for simplicity that $Q_1=[0,r]^d$ and $Q_2=Q_1 + {\rm e}_1$. Let us apply the Chain Lemma~\ref{chle} to the measures $\gamma,\, \gamma'$ and $\gamma_0$, finding the constants $\bar\eps>0$ and $N\in\N$. Now, we select some positive constant $K$ such that
\[
\gamma_0\big( \big\{ (x,y)\in Q_1\times Q_2:\, \rho(x) <K \big\} \big) >0\,,
\]
calling again $\rho$ the density of $\mu$ with respect to the Lebesgue measure. Let us call $\gamma^K_{\rm TEMP}$ the restriction of $\gamma_0$ to the points $(x,y)\in Q_1\times Q_2$ with $\rho(x)<K$, let $m_{\rm TEMP}\ll \bar\eps$ be its mass, and let $\mu^K_{\rm TEMP}$ and $\nu^K_{\rm TEMP}$ be its two marginals. Now, we can repeat verbatim the construction of Step~II in the proof of part~(i) of Theorem~\mref{Main} (with $\eps=1/2$, which is enough for this proof), disintegrating $\mu^K_{\rm TEMP} = \mu_\tau \otimes \alpha$, and defining first the constant $\delta\ll 1$ (to be sent to $0$ at the end) and the function $f_\tau$ for $\tau\in [0,r]^{d-1}$, then the constant $H=2Kr^d/m_{\rm TEMP}$ and the set $Z$ with $\alpha(Z)\leq m_{\rm TEMP}/2$, and finally the sets $L$ and $R$ and the functions $\tilde g$ and $g$. We now call
\[
\gamma^K = \gamma^K_{\rm TEMP} \res \Big\{(x,y)\in Q_1 \times Q_2:\, x=(\sigma,\tau),\, \tau\notin Z \Big\}\,,
\]
and let as usual $\mu^K$ and $\nu^K$ be its marginals, and $m\in [m_{\rm TEMP}/2,m_{\rm TEMP}]$ its total mass. Define now
\[
\xi^1 = g_\# \gamma^K \res \big(L\times Q_2\big) \Big)\,,
\]
and notice carefully that its two marginals are given by
\begin{align*}
\pi_1 \xi^1 (A)= \mu^K_{\rm TEMP} ( A \cap R)\leq \mu^K(A)\,, &&
\pi_2 \xi^1 (B)= \gamma^K_{\rm TEMP} ( L\times B)\leq \nu^K(B) \,;
\end{align*}
as a consequence, we can define the ``remaining measures''
\begin{align*}
\mu_{\rm rem} = \mu^K - \pi_1 \xi^1\,, &&
\nu_{\rm rem} = \nu^K - \pi_2 \xi^1\,,
\end{align*}
whose measure is $\eta := \| \mu_{\rm rem}\| = \| \nu_{\rm rem}\|=\delta m$ by construction. Since $\delta$ can be taken arbitrarily small, we can assume that $\eta < \bar\eps$, so that the Chain Lemma~\ref{chle} provides us with measures $\tilde\gamma=\tilde\gamma_0+\tilde\gamma_\infty$, $\tilde\gamma'$, $\tilde\mu,\, \tilde\nu,\, \mu_A,\, \nu_A$ and $\nu_B$ with $\|\tilde\gamma_0\|=\eta$ and $\|\tilde\gamma_\infty\|=N\eta$ satisfying
\begin{align*}
\mu_A &\leq \mu_0\,, &
\nu_A,\, \nu_B &\leq \nu_0\,, &
\pi_1 \tilde\gamma_0 &= \mu_A \,, &
\pi_2 \tilde\gamma_0 &= \nu_A \,, \\
\pi_1 \tilde\gamma_\infty&= \tilde\mu\,, &
\pi_2 \tilde\gamma_\infty&= \tilde\nu\,, &
\pi_1 \tilde\gamma' &= \tilde\mu + \mu_A\,, &
\pi_2 \tilde\gamma' &= \tilde\nu + \nu_B\,.
\end{align*}
We can then immediately define $\xi^2 = \gamma-\gamma^K - \tilde\gamma_\infty$, and observe that the two marginals of $\xi^1+\xi^2$ are
\begin{align*}
\pi_1 \big(\xi^1+\xi^2\big) = \mu - \mu_{\rm rem} - \tilde\mu\,, &&
\pi_2 \big(\xi^1+\xi^2\big) = \nu - \nu_{\rm rem} - \tilde\nu\,.
\end{align*}
As a consequence, in order to get a competitor transport plan $\xi=\xi^1+\xi^2+\xi^3$, we need a plan $\xi^3$ with marginals $\tilde\mu+\mu_{\rm rem}$ and $\tilde\nu+\nu_{\rm rem}$, so exactly as in the proof of last section we simply define $\xi^3 = \beta_2\circ \tilde\gamma' \circ \beta_1$, being the auxiliary transport plans $\beta_1$ and $\beta_2$ defined as
\begin{align*}
\beta_1 = ({\rm Id},{\rm Id})_\# \tilde\mu+ \mu_{\rm rem} \otimes \mu_A\,, &&
\beta_2 = ({\rm Id},{\rm Id})_\# \tilde\nu+ \nu_B\otimes\nu_{\rm rem} \,.
\end{align*}
Estimating the cost of $\xi^3$ is very simple: as in last section, we only have to observe that every pair $(x,y)$ in the support of $\tilde\gamma'$ satisfies $(y-x) \in t'\BigC$, while the auxiliary plans $\beta_1$ and $\beta_2$ only move points at most of a distance at most $2r\ll 1-t'$, so we get that for all the pairs $(x,y)\in\spt\xi^3$ it is $y-x \in \BigC$. Thus, since $\gamma$ is an optimal transport plan for time $t=1$, we can estimate
\[\begin{split}
\C_1(\gamma) &\leq \C_1(\xi) 
=\C_1(\xi^1) + \C_1(\xi^2)+ \C_1(\xi^3)
\leq \C_1(\xi^1) + \C_1(\gamma-\gamma^K)+ \|h\|_{L^\infty}\|\xi^3\|\\
&= \C_1(\xi^1) + \C_1(\gamma-\gamma^K)+ \|h\|_{L^\infty} (N+1)\,\delta m\,,
\end{split}\]
which implies
\begin{equation}\label{for1}
\C_1\Big(\gamma^K \res \big(L\times Q_2\big)\Big) - \C_1\Big(g_\# \Big(\gamma^K \res \big(L\times Q_2\big) \Big)\Big) \leq \|h\|_{L^\infty} (N+1)\,\delta m\,,
\end{equation}
Let us now estimate the left term in the last inequality as
\begin{equation}\label{for2}\begin{split}
\C_1\Big(\gamma^K &\res \big(L\times Q_2\big)\Big) - \C_1\Big(g_\# \Big(\gamma^K \res \big(L\times Q_2\big) \Big)\Big)
=\iint_{L\times Q_2} c(x,y)-c(\tilde g(x),y) \, d\gamma^K(x,y)\\
&=\iint\limits_{L\times Q_2} h(y-x)-h(y-\tilde g(x)) \, d\gamma^K(x,y)
=\iint\limits_{Q_1\times Q_2} \Chi{L}(x)\big(h(y-x)-h(y-\tilde g(x))\big) \, d\gamma^K(x,y)\,,
\end{split}\end{equation}
and notice that for $\gamma^K$-a.e. $(x,y)$ we have $y-\tilde g(x) = y - x - \varphi(x) {\rm e}_1$, with
\[
\frac{m\delta}{2Kr^{d-1}} \leq \varphi(x) \leq r\,.
\]
Observe now that $y-x$ is in a very small neighborhood of ${\rm e}_1$, hence by convexity of $h$, recalling also that $h(0)=0 < h({\rm e}_1)$, we immediately get that $h$ strictly decreases in the direction $-{\rm e}_1$ in a small neighborhood of ${\rm e}_1$. Thus, we can evaluate
\[
h(y-\tilde g(x)) \leq h(y-x - c \delta{\rm e}_1)\,,
\]
setting $c=m/(2Kr^{d-1})$. Putting this estimate together with~(\ref{for1}) and~(\ref{for2}), we obtain
\[
\iint_{Q_1\times Q_2} \Chi{L}(x)\,\frac{h(y-x) - h(y-x-c\delta {\rm e}_1)}\delta \, d\gamma^K(x,y)\leq \|h\|_{L^\infty} (N+1)\, m\,.
\]
Now, keep in mind that the set $L$ actually depends on the choice of $\delta$; nevertheless, it is obvious from the definition that the set $L$ increases when $\delta$ decreases, and the union of all the sets $L$ for $\delta\to 0$ covers $\gamma^K$ all of $Q_1$. As a consequence, for $\gamma^K$-a.e. $(x,y)$ the function
\[
\delta \mapsto \Chi{L}(x)\,  \frac{h(y-x) - h(y-x-c\delta {\rm e}_1)}\delta
\]
is increasing for $\delta\searrow 0$, and it converges to $-c D_{-{\rm e}_1}h(y-x)$. The Monotone Convergence Theorem gives then
\[
\iint_{Q_1\times Q_2} -D_{-{\rm e}_1} h(y-x)\, d\gamma^K(x,y)\leq \|h\|_{L^\infty} (N+1)\, 2 K r^{d-1} \,.
\]
And finally, $\gamma^K$ is concentrated on pairs $(x,y)$ with $y-x\in \Theta$, so with $D_{-(y-x)}h(y-x)=-\infty$, thus Lemma~\ref{alwrealwin} implies $-D_{-{\rm e}_1} h(y-x) = +\infty$ for $\gamma^K$-a.e. $(x,y)$, and this gives the searched contradiction, so concluding the proof.
\end{proof}

\thanks{{\bf Acknowledgements.} This work has been mostly done while the first author was post-doctoral fellow at the Friederich-Alexander Universit\"at Erlangen-N\"urnberg during the academic year 2014-15, funded by the ERC grant 258685 ``AnOptSetCon''.}

\end{document}